\documentclass{amsart}
\usepackage{amsmath, amsthm, amsfonts, amssymb}
\usepackage{mathrsfs}
\usepackage{microtype}
\usepackage[utf8]{inputenc}
\providecommand{\noopsort[1]{}}
\usepackage{bbm}
\usepackage{dsfont}
\usepackage{ifthen}
\usepackage{nicefrac}
\numberwithin{equation}{section}
\usepackage{a4}
\usepackage[active]{srcltx}
\usepackage{verbatim}
\usepackage{hyperref}
\usepackage{color}

\usepackage[shortlabels]{enumitem}
\setlist{leftmargin=*}
\setlist[1]{labelindent=1.2\parindent}

\allowdisplaybreaks

\newtheorem{thm}{Theorem}[section]
\newtheorem{coro}[thm]{Corollary}
\newtheorem{prop}[thm]{Proposition}
\newtheorem{lm}[thm]{Lemma}

\theoremstyle{remark}
\newtheorem{rmk}[thm]{Remark}

\newtheorem{hyp}[thm]{Hypotheses}
\newtheorem{examp}[thm]{Example}
\newtheorem{examps}[thm]{Example}

\renewcommand{\Re}{{\rm Re}\,}

\renewcommand{\Im}{{\rm Im}\,}

\newcommand{\R}{\mathds{R}}
\newcommand{\C}{\mathds{C}}

\newcommand{\N}{\mathds{N}}

\newcommand{\tnorm}[1]{{|\kern-0.3ex|\kern-0.3ex|#1|\kern-0.3ex|\kern-0.3ex|}}

\newcommand{\la}{\langle}
\newcommand{\ra}{\rangle}

\begin{document}
\title[On a polynomial scalar perturbation in $L^p$-spaces]{On a polynomial scalar perturbation of a Schr\"odinger system in $L^p$-spaces}
\author{A. Maichine}
\address{Dipartimento di Matematica, Universit\`a degli Studi di Salerno, Via Giovanni Paolo II 132, I-84084 Fisciano (SA), Italy}
\email{amaichine@unisa.it}
\author{A. Rhandi}
\address{Dipartimento di Ingegneria dell'Informazione, Ingegneria Elettrica e Matematica Applicata, Università degli Studi di Salerno, Via Ponte Don Melillo 1, 84084 Fisciano (Sa), Italy}
\email{arhandi@unisa.it}
\date{}
\keywords{System of PDE, Schr\"odinger operator, strongly continuous semigroup, domain characterization, kernel estimates, spectrum}
\subjclass[2010]{Primary: 35K40, 47D08; Secondary: 35K08}
\maketitle
%\begin{center}
%\emph{Dedicated to Prof. H. Bouslous \\ on the occasion of his $65^{th}$ birthday}
%\end{center}
\begin{abstract}
In the paper \cite{KLMR} the $L^p$-realization $L_p$ of the matrix Schr\"odinger operator $\mathcal{L}u=div(Q\nabla u)+Vu$ was studied. The generation of a semigroup in $L^p(\R^d,\C^m)$ and characterization of the domain $D(L_p)$ has been established. In this paper we perturb the operator $L_p$ of by a scalar potential belonging to a class including all polynomials and show that still we have a strongly continuous semigroup on $L^p(\R^d,\C^m)$ with domain embedded in $W^{2,p}(\R^d,\C^m)$. We also study the analyticity, compactness, positivity and ultracontractivity of the semigroup and prove Gaussian kernel estimates. Further kernel estimates and asymptotic behaviour of eigenvalues of the matrix Schr\"odinger operator are investigated.
\end{abstract}

\section{Introduction}
While the scalar theory of second-order elliptic operators with unbounded coefficients is by now well developed (cf. \cite{lorenzi17} and the references therein), there is still few research works, at least in the framework of semigroup theory, for systems of parabolic equations with unbounded coefficients. To our knowledge one of the first papers dealing with such kind of systems is \cite{hetal09}.
Subsequently, there were some other publications \cite{aalt, alp16, dlll}. Here the strategy in these references is quite different from
that in \cite{hetal09}. Namely, in \cite{aalt, alp16, dlll} solutions to the parabolic equation are at first constructed in the space of
bounded and continuous functions. Afterwards the semigroup is extrapolated to the $L^p$-scale. This approach cannot give precise information about the domain of the generator of the semigroup.

Recently in \cite{KLMR} a noncommutative Dore--Venni theorem due to
S. Monniaux and J. Pr\"uss \cite{mp97} was used to obtain generation of $C_0$-semigroups for matrix Schr\"odinger operators of type $\mathcal{L}=div(Q\nabla\cdot)+V$ in $L^p$-spaces, where $V$ is a matrix potential whose entries can grow like $|x|^r$ for some $r\in [1,2)$. This approach permits to obtain the maximal inequality
\begin{equation}\label{MIn}
\|div(Q\nabla u)\|_{L^p(\R^d,\C^m)}+\|Vu\|_{L^p(\R^d,\C^m)}\le C\|div(Q\nabla u)+Vu\|_{L^p(\R^d,\C^m)}
\end{equation}
for all $u\in C_c^\infty(\R^d,\C^m)$ and some positive constant $C$ independent of $u$.

An other approach is to use form methods and Beurling-Denny criterion to prove generation of $C_0$-semigroup in $L^p$-spaces. This approach works for Symmetric matrix Schr\"odinger operators, but no information about the domain of the generator can be obtained, see the recent paper \cite{Maichine}.

On the other hand, as we will see in Example \ref{exa-complexe}, there is a relationship between scalar Schr\"odinger operators with complex potentials and matrix Schr\"odinger operators with real matrix potentials. So, our results can be applied to a large class of scalar Schr\"odinger operators with complex potentials.

In this paper we obtain the same generation and regularity results as in \cite{KLMR} for a more general class of potentials whose diagonal entries are polynomials of type $|x|^\alpha$ or even $e^{|x|}$ as well as $|x|^r\log(1+|x|)$, $\alpha,\,r\ge 1$. Our techniques consist in perturbing the above operator $\mathcal{L}$ by a scalar potential $v\in W^{1,\infty}_{loc}(\R^d)$ satisfying $|\nabla v|\lesssim v$, and applying a perturbation theorem due to Okazawa \cite{Okazawa84}. This approach permits us to prove the maximal inequality \eqref{MIn} for such potentials. Furthermore, we obtain sufficient conditions for analyticity of the semigroup and compactness of the resolvent.
\\
Motivated by the kernel estimates existing in literature for scalar Schr\"odinger operators, see for instance \cite{Meta-Spina,Ouha-Rha,Sikora97,MPR06,LorRha, KLR1, KLR2}, we prove Gaussian and other kernel estimates of the obtained semigroups. As a consequence we study in the symmetric case the asymptotic distribution of the eigenvalues.

The paper is organized as follow. In section 2 we state our assumptions, explain our strategy and give some preliminary results. Section 3 is devoted to the main result, that is the generation of a semigroup of the considered operator by applying Okazawa's theorem. In section 4 we study the analyticity, positivity of the semigroup and the compactness of the resolvent. In the fifth section we establish the ultracontractivity property and obtain Gaussian upper estimates for the entries of the matrix kernel. Further kernel estimates are discussed. Section 6 deals with the asymptotic distribution of the eigenvalues of symmetric matrix Schr\"odinger operators.

\textbf{Notation} Let $d,m\geq 1$. By $|.|$ we denote the Euclidean norm on $\C^j$, $j=d,m$ and $\langle \cdot,\cdot \rangle$ the Euclidean inner-product. The set $B(r)=\{x\in\R^d : |x|\leq r \}$ denotes the Euclidean ball of radius $r>0$ and center $0$. For $1\leq p<\infty$, $L^{p}(\R^d,\C^m)$ is the standard Lebesgue space endowed with the norm
\[ \|f\|_p=\left(\int_{\R^d} |f(x)|^p dx \right)^\frac{1}{p}=\left(\int_{\R^d} (\sum_{j=1}^{m}|f_j|^2)^{\frac{p}{2}} dx \right)^\frac{1}{p}, \quad f\in
 L^{p}(\R^d,\C^m).\]
If $1<p<\infty$, $p'$ denotes its conjugate : $1/p+1/p'=1$. Recall that $L^{p'}(\R^d,\C^m)$ is the dual space of $L^{p}(\R^d,\C^m)$ and the duality pairing $\langle \cdot,\cdot \rangle_{p,p'}$ is given by
\[ \langle f,g\rangle_{p,p'}=\int_{\R^d}\langle f(x),g(x)\rangle dx,\qquad (f,g)\in L^{p}(\R^d,\C^m)\times L^{p'}(\R^d,\C^m). \]
By $C_c^k(\R^d,\C^m)$, $k\in\N\cup\{\infty\}$, we denote the space of functions $f:\R^d\rightarrow\C^m$ which are differentiable up to the order k and have compact support. The space $W^{k,p}(\R^d,\C^m)$ denotes the classical Sobolev space of order $k$, that is the space of all functions  $f\in L^{p}(\R^d,\C^m)$ such that the distributional derivative $\partial^\alpha f\in L^{p}(\R^d,\C^m)$ for all $\alpha=(\alpha_1,\dots,\alpha_d)\in\N^d$ with $|\alpha|=\sum_{j=1}^{d}\alpha_j\leq k$. The set $W_{loc}^{k,p}(\R^d,\C^m)$ denotes the space of functions $f\in L^{p}(\R^d,\C^m)$ such that $\chi_{B(r)}f\in W^{k,p}(\R^d,\C^m)$ for all $r>0$, where $\chi_{B(r)}$ is the indicator function of the Euclidean ball $B(r)$.\\
A closed operator $L:D(L)\subset X\to X$ in a Banach space is said to be sectorial if there exists $\theta\in(0,\pi)$ such that $\sigma(L)\subset\Sigma_\theta$ and $$ \sup\{\|\lambda(\lambda-L)^{-1}\|:\lambda\in\C\backslash\Sigma_\theta
\}<\infty ,$$
where $\Sigma_\theta=\{\lambda\in\C: |\arg(\lambda)|<\theta\}$ and $\sigma(L)$ denotes the spectrum of $L$.

We recall that an operator $L$ on $L^p(\R^d,\C^m)$, $1<p<\infty$, is said to be accretive  if, and only if, $\Re\langle Lu, F(u)\rangle_{p,p'}\geq 0$, where
 \[F(u)=\begin{cases}
 \qquad 0,\qquad\qquad\qquad \hbox{\ if}\qquad u=0\quad a.e.\\
 \|u\|_{p}^{2-p} |u|^{p-2}\overline{u},\qquad \hbox{\ else}.
 \end{cases}.
 \]
Moreover, if $(\lambda+L)D(L)=L^{p}(\R^d,\C^m)$, for some $\lambda>0$, we say that $L$ is m-accretive. If $L$ is m-accretive and there exists $M_L\geq 0$ such that
\begin{equation}\label{sectoriality carac}
\Re\langle Lu,F(u) \rangle_{p,p'}\geq M_L |\Im\langle Lu,F(u) \rangle|
\end{equation}
for all $u\in D(L)$, then $L$ is sectorial of angle less than $\pi/2$ and thus $-L$ generates an analytic semigroup.
\section{Hypotheses, notation and preliminaries}

We still use the notation of \cite{KLMR}. In particular, $L_p$ denotes the realization in $L^p(\R^d,\C^m)$ of $\mathcal{L}$, where $\mathcal{L}u=(div(Q\nabla u_j))_{1\le j\le m}+Vu$, for $u=(u_1,\dots,u_m)$ smooth enough. Throughout the paper we assume that the matrices $Q$ and $V$ satisfy the hypotheses of \cite{KLMR}. Namely:
\begin{hyp}\label{Hyp. of KLMR}
%\subsection{Hypotheses}\label{Hyp. of KLMR}
\begin{itemize}
\item[(i)] $Q:\R^d\longrightarrow\R^{d\times d}$ is a Lipschitz function such that $Q(x)=[q_{ij}(x)],\,x\in \R^d,$ is a symmetric matrix, and there exist $\eta_1,\eta_2>0$ with
\begin{equation}\label{ell}
\eta_1 |\xi|^2\leq \langle Q(x)\xi,\xi\rangle \leq \eta_2 |\xi|^2, \qquad x,\xi\in \R^d;
\end{equation}
\item[(ii)] $V=[v_{ij}]:\R^d\longrightarrow\R^{m\times m}$ is a measurable matrix-valued function such that there exists a constant $\beta<0$ with
\begin{equation}\label{diss of V}
 \Re\langle V(x)\xi,\xi\rangle\leq \beta |\xi|^{2},\qquad x \in \R^d, \xi\in\C^m.
\end{equation}
Moreover, assume that $v_{ij}\in W^{1,\infty}_{loc}(\R^d)$ and there exists $\gamma\in[0,\frac{1}{2})$ such that
\begin{equation}\label{Cond. on V}
\sup_{x\in\R^d}|\partial_j V(x)(-V(x))^{-\gamma}|<\infty
\end{equation} for all $j\in\{1,\dots,m\}$.
\end{itemize}
\end{hyp}
\begin{rmk}\label{V-dissip}
Assuming $\beta<0$ in \eqref{diss of V} is not a restriction, since by shifting the potential one can, without loss of generality, assume that such condition is satisfied.
\end{rmk}
The condition \eqref{Cond. on V} allows Lipschitz entries for $V$ or at most, as Example 2.4 of \cite{KLMR} shows, potentials like
$$V(x)=\begin{pmatrix}
0 & 1+|x|^r\\
-(1+|x|^r) & 0
\end{pmatrix}$$ with $r\in[1,2).$
We want to establish now the same results as in \cite{KLMR} for potentials of type
$$\tilde{V}(x)=\begin{pmatrix}
-|x|^\delta & 1+|x|^r\\
-(1+|x|^r) & -|x|^\delta
\end{pmatrix}, \hbox{\ and }\begin{pmatrix}
-|x|^\delta & |x|\\
|x| & -|x|^\delta
\end{pmatrix}$$ for $\delta\ge 1$. To do so we split such a potential into $\tilde{V}=V-vI_m$, where $I_m$ the matrix identity of $\R^m$ and $V$ is a potential satisfying Hypotheses \ref{Hyp. of KLMR} and $v$ is a scalar potential satisfying $|\nabla v|\lesssim v$. Such condition allows all polynomials and potentials like $|x|^r\log(1+|x|)$. Define the operator $\tilde{\mathcal{L}}u=(div(Q\nabla u_j))_{1\le j\le m}+\tilde{V}u$ for $u$ smooth enough. Then $\tilde{\mathcal{L}}$ can be rewritten as $\tilde{\mathcal{L}}u=\mathcal{L}u-vu$, $u\in C_c^\infty(\R^d,\C^m)$. If we denote by $\tilde{L}_p$ the realization on $L^p(\R^d,\C^m)$ of $\tilde{\mathcal{L}}$ with domain $D(\tilde{L}_p)=D(L_p)\cap D(v)$, then we show that $\tilde{L}_p$ generates a contractive strongly continuous semigroup by using the following Okazawa's perturbation theorem, see \cite[Theorem 1.6]{Okazawa}.
%For the estimate (\ref{inequality for compact}) below we refer to \cite[Lemma 1.5]{Okazawa}.
 \begin{thm}\label{Okazawa-1}
 Let $A$ and $B$ be linear m-accretive operators on $X$ such that its dual $X^*$ is uniformly convex. Let $D$ be a core for $A$. Assume that there are nonnegative constants $c$, $a$ and $b$ such that for all $u\in D$ and $\varepsilon>0$,
 \begin{equation}\label{Oka-con}
 \Re\langle Au,F(B_\varepsilon u)\rangle \geq -c\|u\|^2-a\|B_\varepsilon u\|\|u\|-b\|B_\varepsilon u\|^2 ,
 \end{equation}
 where $B_\varepsilon :=B(I+\varepsilon B)^{-1}$ denotes the Yosida approximation of $B$.
 If $t>b$ then $A+tB$ with domain $D(A)\cap D(B)$ is m-accretive and $D(A)\cap D(B)$ is a core for $A$. Furthermore, $A+bB$ is essentially m-accretive on $D(A)\cap D(B)$.
 %Moreover, for every $s>b$ and $u\in D(A)\cap D(B)$, one has
 %\begin{equation}\label{inequality for compact}
 %\|Bu\| \leq \frac{1}{s-b}\|(A+sB)u\|+\left(\frac{a}{s-b}+\sqrt{\frac{c}{s-b}}\right)\|u\|.
 %\end{equation}
 \end{thm}
 Let us introduce some notations which will be used from now on. Since, for each $x\in\R^d$, $Q(x)$ is a symmetric positive nondegenerate matrix, we introduce the $Q(x)$-norm on $\R^d$ and its associated inner-product, given by $$<y,z>_{Q(x)}:=\langle Q(x)y,z\rangle$$ and the corresponding norm is
$$|y|_{Q(x)}=\langle Q(x)y,y\rangle^{\frac{1}{2}}.$$

Before starting applying Okazawa's theorem we need to show that $C_c^\infty(\R^d,\C^m)$ is a core for $L_p$, for all $1<p<\infty$.\\
 We first state a lemma which gives a generalisation of the Stampacchia theorem concerning the weak derivative of the absolute value function, see \cite[Lemma 7.6]{Gilbarg}.
\begin{lm}\label{1st lemma of sec 3}
 Let $1<p<\infty$ and $u=(u_1(\cdot),\dots,u_m(\cdot))\in W^{1,p}(\R^d,\C^m)$. Then,  $|u|\in W^{1,p}(\R^d)$ and
 \begin{equation}\label{gradient of |.|=}
 \nabla|u|=\frac{1}{|u|}\sum_{j=1}^{m}\Re(\bar{u}_j\nabla u_j)\chi_{\{u\ne 0\}}.
 \end{equation}
 Moreover,
 \begin{equation}\label{gradient |.| <...}
 |\nabla |u||_Q^2 \leq\displaystyle\sum_{j=1}^{d}|\nabla u_j|_Q^2.
 \end{equation}
 \end{lm}
 \begin{proof}
 (i) Let $\varepsilon>0$ and define $a_\varepsilon(u)=(\displaystyle\sum_{j=1}^{m}u_j^2+\varepsilon^2)^{\frac{1}{2}}-\varepsilon$. Then, $a_\varepsilon(u)\in W^{1,p}(\R^d)$ and \[\nabla a_\varepsilon(u)=\dfrac{\displaystyle\sum_{j=1}^{m}\Re(\bar{u}_i\nabla u_j)}{(\displaystyle\sum_{j=1}^{m}u_j^2+\varepsilon^2)^{\frac{1}{2}}}.\]
 We have the following pointwise convergence: $a_\varepsilon(u)\underset{\varepsilon\to 0}{\longrightarrow} |u|$ and
 $$\nabla a_\varepsilon(u)\underset{\varepsilon\to 0}{\longrightarrow} \frac{1}{|u|}\displaystyle\sum_{j=1}^{m}\Re(\bar{u}_j\nabla u_j) \chi_{\{u\neq 0\}}.$$
  Now, since $a_\varepsilon(u)=\dfrac{|u|^2}{(\displaystyle\sum_{j=1}^{m}u_j^2+\varepsilon^2)^{\frac{1}{2}}+\varepsilon}\le |u|$ and by Young inequality $|\nabla a_\varepsilon(u)|\le |\nabla u|$, thus the dominated convergence theorem yields $|u|\in W^{1,p}(\R^d)$ and \eqref{gradient of |.|=}.\\
  (ii) One knows that
  \[ \Re(\bar{u}_j\nabla u_j)=\frac{1}{2}\nabla |u_j|^2\,\hbox{ and, by }\eqref{gradient of |.|=},\,|u|\nabla|u|=\frac{1}{2}\nabla|u|^2=\frac{1}{2}\sum_{j=1}^{m}\nabla|u_j|^2 .\]
  Hence, by the Cauchy-Schwartz inequality, it follows that
  \begin{align*}
  |u||\nabla|u||_Q &= \frac{1}{2}\left|\sum_{j=1}^{m}\nabla|u_j|^2\right|_Q \\
  &\leq \frac{1}{2}\sum_{j=1}^{m}|\nabla|u_j|^2|_Q \\
  &\leq \sum_{j=1}^{m}|\Re (u_j \nabla u_j)|_Q \\
  &\leq \sum_{j=1}^{m}| u_j \nabla u_j|_Q \\
  &\leq |u|\left(\sum_{j=1}^{m}|\nabla u_j|_Q^2\right)^{\frac{1}{2}}.
  \end{align*}
  Thus, $|\nabla |u||_Q^2 \leq\displaystyle\sum_{j=1}^{d}|\nabla u_j|_Q^2$.
  \end{proof}
\begin{rmk}\label{adjoint-lp}
We note that the adjoint $V^\ast$ also satisfies
Hypotheses \ref{Hyp. of KLMR}, except for the
fact that instead of the boundedness of
$\partial_j V^\ast (-V^\ast)^{-\gamma}$, we have the boundedness
of $(-V^\ast)^{-\gamma} \partial_jV^\ast$.
However, an inspection of the proofs in \cite[Theorem 3.2 and Corollary 3.3]{KLMR} shows that they
remain valid also under this assumption, whence we obtain the same results for $L_p^\ast$, the adjoint of $L_p$ with $p\in (1,+\infty)$.
\end{rmk}
Using the above remark we prove now that $C_c^\infty(\R^d,\C^m)$ is a core for $L_p$.
\begin{prop}\label{Coincidence of domains}
Let us assume Hypotheses \ref{Hyp. of KLMR}. Then for any $1<p<\infty$
$$D(L_p)=\{u\in L^p(\R^d,\C^m)\cap W^{2,p}_{loc}(\R^d,\C^m) : \mathcal{L}u\in L^p(\R^d,\C^m)\}=: D_{p,max}(\mathcal{L}),$$
and $C_c^\infty(\R^d,\C^m)$ is a core for $L_p$.
\end{prop}
\begin{proof}
From \cite[Corollary 3.3]{KLMR} we know that
$$D(L_p)=\{u\in W^{2,p}(\R^d,\C^m): Vu\in L^p(\R^d,\C^m)\}.$$
  So, it is obvious that $D(L_p)\subseteq D_{max}(\mathcal{L})$. In order to prove the other inclusion it suffices to show that $\lambda-\mathcal{L}$ is injective on $D_{max}(\mathcal{L})$, for some $\lambda>0$. To this purpose let $u\in D_{max}(\mathcal{L})$ such that $(\lambda-\mathcal{L})u=0$.
  Since the coefficients of the operator $\mathcal{L}$ are real, one can assume that $u$ is real valued function. Consider $\zeta\in C_c^\infty(\R^d)$ satisfying $\chi_{B(1)}\leq\zeta\leq\chi_{B(2)}$ and define $\zeta_n(\cdot)=\zeta(\cdot/n)$ for all $n\in\N$. Assume first that $p\geq 2$. Multiplying $(\lambda-\mathcal{L})u$ by $\zeta_n^2 |u|^{p-2}u$ and integrating over $\R^d$ one obtains
 \begin{eqnarray*}
 0&=&\lambda\int_{\R^d}\zeta_n^2(x)|u(x)|^p dx+\int_{\R^d}\sum_{j=1}^{m}\langle Q\nabla u_j,\nabla(|u|^{p-2}u_j\zeta_n^2)\rangle dx\\
 &-&\int_{\R^d}\langle V(x)u(x),u(x)\rangle |u(x)|^{p-2}\zeta_n^2(x)dx\\
 &\geq & \lambda\int_{\R^d}\zeta_n^2(x)|u(x)|^p dx+\int_{\R^d}|u(x)|^{p-2}\zeta_n^2(x)\sum_{j=1}^{m}\langle Q(x)\nabla u_j(x),\nabla u_j(x)\rangle dx\\
 &+& 2\int_{\R^d}\sum_{j=1}^{m}|u(x)|^{p-2}u_j(x)\zeta_n(x)\langle Q(x)\nabla u_j(x),\nabla \zeta_n(x)\rangle dx\\
 &+&(p-2)\int_{\R^d}|u(x)|^{p-2}\zeta_n^2(x)\langle Q(x)\nabla |u|(x),\nabla |u|(x)\rangle dx\\
 &\geq &\lambda\int_{\R^d}\zeta_n^2(x)|u(x)|^p dx+\int_{\R^d}|u(x)|^{p-2}\zeta_n^2(x)\sum_{j=1}^{m}\langle Q(x)\nabla u_j(x),\nabla u_j(x)\rangle dx\\
  &+& 2\int_{\R^d}|u(x)|^{p-2}\zeta_n(x)\sum_{j=1}^{m}u_j(x)\langle Q(x)\nabla u_j(x),\nabla \zeta_n(x)\rangle dx\\
 &\geq & \lambda\int_{\R^d}\zeta_n^2(x)|u(x)|^p dx-\int_{\R^d}|u(x)|^{p-2}\sum_{j=1}^{m}\langle Q(x)\nabla \zeta_n (x),\nabla \zeta_n (x)\rangle u_j^2(x)\\
 &\geq & \lambda\int_{\R^d}\zeta_n^2(x)|u(x)|^p dx-\frac{\eta_2\|\nabla\zeta\|_\infty^2}{n^2}\int_{\R^d}|u(x)|^p dx.
 \end{eqnarray*}
  Letting $n$ goes to $\infty$ one obtains $0\geq \lambda\|u\|^p$, and hence $u=0$. For the case $p<2$, one can multiply by $\zeta_n(|u|^2+\varepsilon)^{\frac{p-2}{2}}u$, $\varepsilon>0$, instead of $\zeta_n|u|^{p-2}u$ and repeat the same calculus to obtain the result by tending $\varepsilon$ to $0$.\\
 In order to prove that $C_c^\infty(\R^d,\C^m)$ is a core for $L_p$ it suffices to show that $(\lambda-L_p)C_c^\infty(\R^d,\C^m)$ is dense in $L^p(\R^d,\C^m)$ for some $\lambda>0$. For this purpose let $f\in L^{p'}(\R^d,\C^m)$ such that $\langle (\lambda-L_p)\varphi, f\rangle=0$ for all $\varphi\in C_c^\infty(\R^d,\C^m)$. Hence,
 \[ \int_{\R^d}\langle\varphi(x),(\lambda+V^*(x)f(x)\rangle dx=\int_{\R^d}\langle div(Q\nabla\varphi)(x),f(x)\rangle dx,\qquad\forall\varphi\in C_c^\infty(\R^d,\C^m). \]
 Standard elliptic regularity yields that $f\in W^{2,p'}_{loc}(\R^d,\C^m)$ and
 \[ \int_{\R^d}\langle\varphi(x),(\lambda+V^*(x)f(x)\rangle dx=\int_{\R^d}\langle \varphi(x),div(Q\nabla f)(x)\rangle dx,\qquad\forall\varphi\in C_c^\infty(\R^d,\C^m). \]
 Then, $(\lambda+V^*)f=div(Q\nabla f)$ a.e. Hence, by Remark \ref{adjoint-lp} and the above characterization of the domain, it follows that $f\in D(L_p^*),$  $(\lambda-L_p^*)f=0$ and thus $f=0$. This ends the proof.
\end{proof}

\section{Generation of semigroup}
In this section we assume that Hypotheses \ref{Hyp. of KLMR} are satisfied.

Consider $0\le v\in W^{1,\infty}_{loc}(\R^d)$ and define on $L^{p}(\R^d,\C^m)$ the multiplication operator with its maximal domain
\begin{equation}\label{scalar-mult-oper}
B_p u=vu,\,\,u\in D(B_p)=\{u\in L^{p}(\R^d,\C^m) : vu\in L^{p}(\R^d,\C^m) \}.
\end{equation}
It is easy to see that $B_p$ is m-accretive on $L^{p}(\R^d,\C^m)$. Define, for $\varepsilon>0$ and $u\in L^{p}(\R^d,\C^m)$, the function $v_\varepsilon :=v(1+\varepsilon v)^{-1}$ and $B_{p,\varepsilon}u :=v_\varepsilon u$. The operator $B_{p,\varepsilon}$ is the Hille-Yosida approximation of $B_p$.\\
%Our goal is to perturb $L_p$ by the diagonal multiplication operator $-B_p$ and use a perturbation theorem due to N. Okazawa,
 %Let $u\in C_c^\infty(\R^N,\R^m)$, recall that, for each $x\in\R^N$, if $u(x)=(u_1(x),\dots,u_m(x))$. Then, $|u(x)|=(\displaystyle\sum_{j=1}^{m}u_j^2(x))^{\frac{1}{2}}$ and $\|u\|_p=\|u\|_{L^{p}(\R^N,\R^m)}=\left(\int_{\R^N}|u(x)|^p dx\right)^{\frac{1}{p}}$.

%Assume that the hypotheses \textbf{(H)} hold and that the constant $\beta$ of (\ref{diss of V}) is negative. Then, $B_p$ is dissipative, and this implies that $-A_p$ is an m-accretive operator.
%For $z\in\R^N$, we denote by $\Re z$ its real part and by $\Im z$ its imaginary part.\\
Fix $1<p<\infty$, $\varepsilon>0$ and $u\in C_c^\infty(\R^d,\C^m)$, and set $R :=v_\varepsilon ^{p-1}$.\\
We propose now to prove \eqref{Oka-con} for the operators $-L_p$ and $B_p$. We start by approximating the left hand side of \eqref{Oka-con} as follow
\begin{equation}\label{approximation of dual product}
  \Re\langle-L_p u,|B_{p,\varepsilon}u|^{p-2}B_{p,\varepsilon}u\rangle_{p,p'}=\lim_{\delta\rightarrow 0}P_\delta,
 \end{equation}
 where \[ P_\delta:= \Re\langle-L_p u,Ru_{\delta}^{p-2}u\rangle_{p,p'} ,\]
 and
 \[ u_\delta=\begin{cases}
   (|u|^2+\delta)^{\frac{1}{2}},\qquad \hbox{\ if}\quad 1<p<2,\\
   \quad |u|, \qquad\qquad \hbox{\ if}\quad p\geq 2.
   \end{cases} .\]
   Noting that  $u_\delta ^{p-2}u\longrightarrow |u|^{p-2}u$, as $\delta\rightarrow 0$, in $L^{p'}(\R^d,\C^m)$. So, the convergence in \eqref{approximation of dual product} follows easily.\\
The following lemma yields a lower estimate for $P_\delta$.
\begin{lm}\label{enonce-lm}
Let $\delta>0$. One has, for $p\geq 2$,
\begin{eqnarray}\label{lm p>2}
P_\delta &\geq & (p-1)\int_{\R^d}|\nabla|u(x)||_{Q(x)}^2 u_\delta ^{p-2}(x)R(x)dx \nonumber \\
& +& \frac{1}{2}\int_{\R^d}\langle \nabla|u|^2 (x),\nabla R(x) \rangle_{Q(x)} u_\delta ^{p-2}(x)dx,
\end{eqnarray}
and for $1<p<2$
\begin{eqnarray}\label{lm p<2}
P_\delta &\geq & (p-1)\int_{\R^d}\sum_{j=1}^{m}|\nabla u_j (x)|_{Q(x)}^2 u_\delta ^{p-2}(x)R(x)dx\nonumber \\
& +& \frac{1}{2}\int_{\R^d}\langle \nabla|u|^2 (x),\nabla R(x) \rangle_{Q(x)} u_\delta ^{p-2}(x)dx.
\end{eqnarray}
\end{lm}
\begin{proof}
Applying integration by part formula and taking into account \eqref{diss of V} one obtains
\begin{align*}
P_\delta &= -\Re\langle L_p u,Ru_{\delta}^{p-2}u\rangle_{p,p'}\\
&= -\Re\int_{\R^d}\la div(Q\nabla u)(x),R(x)u_\delta ^{p-2}(x)u(x)\ra dx-\int_{\R^d}\Re\la V(x)u(x),u(x)\ra R(x)u_\delta ^{p-2}(x)dx\\
&\ge -\sum_{j=1}^{m}\int_{\R^d}\Re\left(div(Q\nabla u_j)(x)\bar{u}_j (x)R(x)u_\delta ^{p-2}(x)\right)dx\\
&= \sum_{j=1}^{m}\int_{\R^d}\Re\langle \nabla u_j(x),\nabla(Ru_\delta ^{p-2}u_j)(x)\rangle_{Q(x)}dx\\
&= \sum_{j=1}^{m}\int_{\R^d} |\nabla u_j(x)|_{Q(x)}^{2} R(x)u_{\delta}^{p-2}(x)dx + \sum_{j=1}^{m}\int_{\R^d}\Re\langle\nabla u_j (x),\nabla R(x)\rangle_{Q(x)}\bar{u}_j (x)u_{\delta}^{p-2}(x)dx\\
&+ \frac{p-2}{2}\sum_{j=1}^{m}\int_{\R^d}\Re\langle\nabla u_j (x),\nabla |u|^2 (x) \rangle_{Q(x)}\bar{u}_j(x)R(x)u_{\delta}^{p-4}(x)dx\\
&= \sum_{j=1}^{m}\int_{\R^d} |\nabla u_j(x)|_{Q(x)}^{2} R(x)u_{\delta}^{p-2}(x)dx + \frac{1}{2}\sum_{j=1}^{m}\int_{\R^d}\langle\nabla |u_j|^2 (x),\nabla R(x)\rangle_{Q(x)}u_{\delta}^{p-2}(x)dx\\
&+ \frac{p-2}{4}\sum_{j=1}^{m}\int_{\R^d}\langle\nabla |u_j|^2 (x),\nabla |u|^2 (x) \rangle_{Q(x)}R(x)u_{\delta}^{p-4}(x)dx\\
&= \sum_{j=1}^{m}\int_{\R^d} |\nabla u_j(x)|_{Q(x)}^{2} R(x)u_{\delta}^{p-2}(x)dx + \frac{1}{2}\int_{\R^d}\langle\nabla |u|^2 (x),\nabla R(x)\rangle_{Q(x)}u_{\delta}^{p-2}(x)dx\\
&+ \frac{p-2}{4}\int_{\R^d}\langle\nabla |u|^2 (x),\nabla |u|^2 (x) \rangle_{Q(x)}R(x)u_{\delta}^{p-4}(x)dx\\
&= \sum_{j=1}^{m}\int_{\R^d} |\nabla u_j(x)|_{Q(x)}^{2} R(x)u_{\delta}^{p-2}(x)dx+(p-2)\int_{\R^d}|\nabla |u|(x)|^2_{Q(x)}|u(x)|^2 R(x)u_{\delta}^{p-4}(x)dx\\
&+ \frac{1}{2}\int_{\R^d}\langle\nabla |u|^2 (x),\nabla R(x)\rangle_{Q(x)}u_{\delta}^{p-2}(x)dx.
\end{align*}
 Taking into account \eqref{gradient |.| <...} and that $u_\delta=|u|$ when $p\geq 2$, and $u_\delta\ge |u|$ and $p-2<0$ when $1<p<2$, one obtains \eqref{lm p>2} and \eqref{lm p<2}.
\end{proof}
The second step is a consequence of Lemma \ref{enonce-lm} and a modification of \cite[Proposition 3.2]{Okazawa}.
%We next give a second inequality verified by $P_\delta$ and an inequality involving the dual product beetwen $-L_p$ and $B_p$
\begin{lm}
Let $\delta>0$. One has
\begin{equation}\label{est-P-delta}
P_\delta\geq-\dfrac{1}{4(p-1)}\int_{\R^d}u_\delta^{p-2}(x)|u(x)|^2\dfrac{|\nabla R(x)|_{Q(x)}^2}{R(x)}dx,
\end{equation}
and hence
\begin{equation}\label{est-dual product}
\Re\langle -L_p u,|B_{p,\varepsilon}u|^{p-2}B_{p,\varepsilon}u\rangle_{p,p'}\geq -\dfrac{1}{4(p-1)}\int_{\R^d}|u(x)|^p\dfrac{|\nabla R(x)|_{Q(x)}^2}{R(x)}dx.
\end{equation}
\end{lm}
\begin{proof}
Let $c=\frac{1}{2(p-1)}$. One can use, for $p\geq 2$, the inequality
\[ \int_{\R^d}u_\delta^{p-2}(x)\left|\sqrt{R(x)}\nabla|u|(x)+\frac{c}{\sqrt{R(x)}}|u(x)|\nabla R(x)\right|_{Q(x)}^2 dx\geq 0 \]
which implies that
\begin{eqnarray*}
0 &\leq &\int_{\R^d}u_\delta^{p-2}(x)R(x)|\nabla|u|(x)|_{Q(x)}^2 dx +c\int_{\R^d}u_\delta^{p-2}(x)\langle\nabla|u|^2(x),\nabla R(x) \rangle_{Q(x)}\\
 & &\quad + c^2\int_{\R^d}u_\delta^{p-2}(x)|u(x)|^2\dfrac{|\nabla R(x)|_{Q(x)}^2}{R}dx.
\end{eqnarray*}
Multiplying by $p-1$ and using \eqref{lm p>2} one obtains \eqref{est-P-delta}. On the other hand, for $1<p<2$, one can use the inequality
\[ \int_{\R^d}u_\delta^{p-2}(x)\sum_{j=1}^{m}\left|\sqrt{R(x)}\nabla u_j(x) +\frac{cu_j(x)}{\sqrt{R(x)}}\nabla R(x) \right|_{Q(x)}^2 dx\geq 0,\]
arguing similarly as above and using \eqref{lm p<2} one obtains \eqref{est-P-delta}. Estimate \eqref{est-dual product} follows now by letting $\delta\to 0$ in \eqref{est-P-delta}.
\end{proof}
We prove now the main theorem of this section,
%which is a vector-valued operator version of \cite[Theorem 3.3]{Okazawa} and \cite[Theorem 2.1]{Okazawa84}.
\begin{thm}\label{thm. gen. of. s.g.}
%Let $d\in C^1(\R^N)$ such that $d>0$ and $C_p$ the multiplication operator by $d$ defined in \eqref{scalar-mult-oper}.
Assume that Hypotheses \ref{Hyp. of KLMR} hold and there exist nonnegative constants $a$ and $b$ such that
\begin{equation}\label{estimate grad v_eps}
|\nabla v_\varepsilon (x)|_{Q(x)}^2\leq a(v_\varepsilon (x))^2+b(v_\varepsilon (x))^3
\end{equation}
for all $\varepsilon>0$ and $x\in\R^d$. Then ${L}_p-sB_p$ with domain $D(L_p)\cap D(B_p)$ generates a contractive $C_0$-semigroup in $L^p(\R^d,\C^m)$ for each $s>\dfrac{(p-1)b}{4}$.
\end{thm}
\begin{proof}
We will show the following inequality
\begin{equation}\label{est-to-prouve}
\Re\langle-L_p u,\|B_{p,\varepsilon}u\|^{2-p}|B_{p,\varepsilon}u|^{p-2}B_{p,\varepsilon}u \rangle_{p,p'} \geq -\frac{p-1}{4}a\|B_{p,\varepsilon}u\|\|u\|-\frac{p-1}{4}b\|B_{p,\varepsilon}u\|^2,
\end{equation}
for every $u\in C_c^\infty(\R^d,\C^m)$. Since from Proposition \ref{Hyp. of KLMR} we know that $C_c^\infty(\R^d,\C^m)$ is a core for $L_p$ and $-L_p$ is m-accretive, see \cite[Corollary 3.3]{KLMR}, we conclude by Theorem $\ref{Okazawa-1}$. According to \eqref{est-dual product} and \eqref{estimate grad v_eps}, one has
\begin{eqnarray*}
 & & \Re\langle-L_p u,\|B_{p,\varepsilon}u\|^{2-p}|B_{p,\varepsilon}u|^{p-2}B_{p,\varepsilon}u\rangle_{p,p'} \\
&\ge & -\dfrac{1}{4(p-1)}\|B_{p,\varepsilon}u\|^{2-p}\int_{\R^d}|u(x)|^p\dfrac{|\nabla R(x)|_{Q(x)}^2}{R(x)}dx\\
&\ge &-\dfrac{p-1}{4}\|B_{p,\varepsilon}u\|^{2-p}\int_{\R^d}|u(x)|^p (v_\varepsilon(x))^{p-3}|\nabla v_\varepsilon(x)|_{Q(x)}^2 dx\\
&\ge & -\dfrac{p-1}{4}a\|B_{p,\varepsilon}u\|^{2-p}\int_{\R^d}|u(x)|^p |v_\varepsilon(x)|^{p-1}dx\\
& &-\dfrac{p-1}{4}b\|B_{p,\varepsilon}u\|^{2-p}\int_{\R^d}|u(x)|^p |v_\varepsilon(x)|^{p}dx.
\end{eqnarray*}
Taking into account that $|B_{p,\varepsilon}(x)u|=|v_\varepsilon(x)||u(x)|$ and using H\"older's inequality, one obtains \eqref{est-to-prouve}. By Theorem \ref{Okazawa-1}, one conclude that $-(L_p-s B_p)$ is m-accretive for $s>\dfrac{(p-1)b}{4}$.
%Let us prove now the last statement. Since $A_p$ is sectorial according to Theorem \ref{1st theorem} then, exists $c_p\geq 0$ such that
%\[ \Re\langle-A_p u,F(u)\rangle_{p,p'}\geq c_p |\Im\langle-A_p u,F(u)\rangle_{p,p'}|. \]
%Hence, since $d>0$, for every $s>\dfrac{(p-1)b}{4}>0$
%\begin{eqnarray*}
%\Re\langle(-A_p +sC_p)u,F(u)\rangle_{p,p'}&=& \Re\langle-A_p u,F(u)\rangle_{p,p'}+s\Re\langle C_p u,F(u)\rangle_{p,p'}\\
%&\geq & \Re\langle-A_p u,F(u)\rangle_{p,p'}\\
%&\geq & c_p|\Im\langle-A_p u,F(u)\rangle_{p,p'}|\\
%&=&c_p|\Im\langle(-A_p+sC_p) u,F(u)\rangle_{p,p'}|,
%\end{eqnarray*}
%which shows that $A_p -sC_p$ is sectorial too.
\end{proof}
 \begin{coro}\label{coro gen. of s.g.}
 Assume that there exists $c>0$ such that
 \begin{equation}\label{grad v< v}
 |\nabla v(x)|\le c\, v(x)
 \end{equation}
  for a.e. $x\in\R^d$, and Hypotheses \ref{Hyp. of KLMR} are satisfied. Then, $\tilde{L}_p=L_p+ B_p$ with domain $D(L_p)\cap D(B_p)$ generates a contractive $C_0$-semigroup on $L^p(\R^d,\C^m)$.
 \end{coro}
 \begin{proof}
 One has $\nabla v_\varepsilon=\nabla(v(1+\varepsilon v)^{-1})=(1+\varepsilon v)^{-2}\nabla v$, which implies $|\nabla v_\varepsilon|\le c v_\varepsilon$. Thus \eqref{estimate grad v_eps} is verified with $a=c^2$ and $b=0$. Thus $L_p-sB_p$ with domain $D(L_p)\cap D(B_p)$ is m-accretive for all $s>0$. In particular, $\tilde{L}_p$ generates a contractive $C_0$-semigroup on $L^p(\R^d,\R^m)$.
 \end{proof}

\section{Further properties of the semigroup}
Let $Q$, $V$ satisfy Hypotheses \ref{Hyp. of KLMR} and $0\le v\in W^{1,\infty}_{loc}(\R^d)$ such that \eqref{grad v< v} holds. Consider $\tilde{V}=V-v I_m$ and denote by
%$(\tilde{V}_p,D(\tilde{V}_p))$ the muliplication operator by $\tilde{V}$ in $L^p(\R^d,\R^m)$ with its maximal domain, and
$\{S_p(t)\}_{t\ge 0}$ and $\{\tilde{S}_p(t)\}_{t\ge 0}$ the contractive $C_0$-semigroups generated respectively by $L_p$ and $\tilde{L}_p$.\\
\medskip
We first prove consistency and characterize positivity of the semigroup $\{\tilde{S}_p(t)\}_{t\ge 0}$.
\begin{prop}\label{prop. positivity+consistence}
\begin{itemize}
\item[(1)] The semigroups $\{\tilde{S}_p(t)\}_{t\ge 0}$, $1<p<\infty$, are consistent.\\
\item[(2)] $\{\tilde{S}_p(t)\}_{t\ge 0}$ is positive if, and only if, the off-diagonal entries of $\tilde{V}$ are nonnegative, i.e. $v_{ij}\ge 0$ for all $i\neq j$.
\end{itemize}
\end{prop}
\begin{proof}
\begin{itemize}
\item[(1)] The first assertion follows from the consistency of $\{S_p(t)\}$ and $\{e^{-tv}\}$ and by the Trotter-Kato product formula
\begin{equation}\label{Trotter Kato}
\tilde{S}(t)=\lim_{n\to\infty}\left(e^{-tv/n}S_p(t/n)\right)^n.
\end{equation}
\item[(2)] Since the off-diagonal entries of $V$ and $\tilde{V}$ are equal, it follows that, if they are nonnegative, then $\{S_p(t)\}$ is positive, see \cite[Proposition 4.1]{KLMR}. Now, applying \eqref{Trotter Kato} the positivity of $\{\tilde{S}_p(t)\}$ follows. Conversely, assume that $\{\tilde{S}_p(t)\}$ is positive. By the minimum principle for positive semigroups, see \cite[C-II Proposition 1.7]{Yellow book}, one obtains that $\langle\tilde{L}_p(\varphi e_j),\varphi e_i\rangle\ge 0$ for all $i\neq j\in\{1,\dots,m\}$ and $0\le\varphi\in C_c^\infty(\R^d)$, where $e_j$ denotes the vector of the canonical orthonormal basis of $\R^m$. This implies that $$\int_{\R^d}v_{ij}(x)\varphi^2(x) dx\ge 0$$ for every $\varphi\in C_c^\infty(\R^d)$. Thus, $v_{ij}\ge 0$ almost everywhere.
\end{itemize}
\end{proof}
\subsection{Compactness and spectrum}
The aim of this paragraph is to give conditions under which the resolvent operator of $\tilde{L}_p$ is compact. As a consequence one deduces that the spectrum of $\tilde{L}_p$ consists only on eigenvalues and is discrete.
To this purpose we first state a similar result as Proposition \ref{Coincidence of domains}, which yields the coincidence between the natural and maximal domains of $\tilde{L}_p$. We omit the proof, since it is the same as the one of
Proposition \ref{Coincidence of domains} where one has to substitute $V$ with $\tilde{V}$.
\begin{prop}
Let $1<p<\infty$. Assume that Hypotheses \ref{Hyp. of KLMR} and \eqref{grad v< v} are satisfied. Then
\begin{equation*}
D(\tilde{L}_p)=\{u\in W^{2,p}_{loc}(\R^d,\C^m) : \tilde{L}u=div(Q\nabla u)+\tilde{V}u\in L^p(\R^d,\C^m)\}=:D_{p,max}(\tilde{\mathcal{L}}).
\end{equation*}
%Moreover, as $W^{2,p}\cap D(\tilde{V}_p)$ is an intermediate domain, then $D(\tilde{L}_p)=W^{2,p}\cap D(\tilde{V}_p)=D_{p,max}(\tilde{\mathcal{L}})$.
\end{prop}
\begin{rmk}\label{rmk equiv. of norms}
An important consequence of the above proposition is the equivalence between the graph norm of $\tilde{L}_p$ and the norms $u\mapsto |||u|||_1:=\|L_p u\|_p+\|vu\|_p$ and $u\mapsto |||u|||_2:=\|u\|_{2,p}+\|\tilde{V}u\|_p$. This information will be very useful for the proof of compactness.
%Here we state the compactness result
\end{rmk}
\begin{prop}\label{Prop Compactness}
Let $1<p<\infty$. Assume that Hypotheses \ref{Hyp. of KLMR} and \eqref{grad v< v} are satisfied. Assume further that one of the following assertions holds true:
\begin{itemize}
\item[(i)] there exists $\rho:\R^d\to\R^+$ measurable such that $\lim_{|x|\to\infty}\rho(x)=\infty$ and
\begin{equation}\label{comp. cond.}
|V(x)\xi|\ge \rho(x)|\xi|,\qquad \forall x\in\R^d, \xi\in\C^m;
\end{equation}
\item[(ii)] $\lim_{|x|\to\infty}v(x)=\infty$.
\end{itemize}
Then, $\tilde{L}_p$ has a compact resolvent in $L^p(\R^d,\C^m)$. Therefore its spectrum is independent of $p \in (1,\infty)$ and consists of eigenvalues only.
\end{prop}
\begin{proof}
Let us define $\tilde{\rho}=\rho$ if (i) is satisfied and $\tilde{\rho}=v$ if (ii) is satisfied (and $\tilde{\rho}=\min(\rho,v)$ if both (i) and (ii) are satisfied). Then $|\tilde{V}(x)\xi|\ge \tilde{\rho}(x)|\xi|$ for all $x\in\R^d$ and $\xi\in\C^m$. Indeed, fix $x\in\R^d$ and $\xi\in\C^m$. Then, by \eqref{diss of V},
\[ |\tilde{V}(x)\xi|^2=|V(x)\xi|^2-2v(x)\Re \langle V(x)\xi,\xi\rangle+v(x)^2|\xi|^2\ge |V(x)\xi|^2+v(x)^2|\xi|^2. \]
%By Cauchy-Schwartz inequality, \eqref{comp. cond.} implies $|V(x)\xi|\ge \rho(x)|\xi|$.
 Thus in both two cases $|\tilde{V}(x)\xi|\ge \tilde{\rho}(x)|\xi|$, and of course in both cases $\lim_{|x|\to\infty}\tilde{\rho}(x)=\infty$. Thus,
\begin{equation}
\|\tilde{V}_pu\|_p^p \geq \int_{\R^d}\tilde{\rho} (x)^p|u(x)|^p\, dx
\label{freccia}
\end{equation}
for every $u\in D(\tilde{L}_p)$.
Let us now prove that the closed unit ball of $D(\tilde{L}_p)$ is compact in $L^p(\R^d;\C^m)$. To do so, let $u\in D(\tilde{L}_p)$ such that $\|u\|_{D(\tilde{L}_p)}\le 1$. Then, by Remark \ref{rmk equiv. of norms},  $\|\tilde{V}_pu\|_p\le |||u|||_2\le C$, for some constant $C\ge 1$.
Let $\varepsilon>0$ and $R>0$ sufficiently large so that $\varepsilon\tilde{\rho}(x)\ge C$ for all $x\in \R^d\setminus B(R)$. Then, from \eqref{freccia}, we deduce that
\begin{align*}
\int_{\R^d\setminus B(R)}|u(x)|^p\, dx\le &\frac{\varepsilon^p}{C^p}\int_{\R^d\setminus B(R)}\tilde{\rho}(x)^p|u(x)|^p\,dx\\
\le &\frac{\varepsilon^p}{C^p}\int_{\R^d}\tilde{\rho}(x)^p|u(x)|^p\,dx
\le \frac{\varepsilon^p}{C^p}\|V_pu\|_p^p\le \varepsilon^p.
\end{align*}
Since $D(\tilde{L}_p)$ is locally continuously embedded in
%coincides with
$W^{2,p}(B(R);\C^m)$, which is compactly embedded into $L^p(B(R); \C^m)$, by the Rellich-Kondarov theorem, one can get a finite sequence of functions $g_1,\ldots,g_k\in L^p(B(R);\C^m)$ such that, for
every $u$ in the unit ball of $D(\tilde{L}_p)$, there
exists $j\in \{1,\ldots, k\}$ such that
\begin{eqnarray*}
\int_{B(R)}|u(x)-g_j(x)|^p\,dx\le\varepsilon^p.
\end{eqnarray*}
Denoting the trivial extension of $g_j$ to $\R^d$ by $\bar g_j$, one has
\begin{align*}
\int_{\R^d}|u(x)-\bar g_j(x)|^p\,dx=\int_{B(R)}|f(x)-g_j(x)|^p\,dx
+\int_{\R^d\setminus B(R)}|u(x)|^p\,dx\le 2\varepsilon^p.
\end{align*}
This shows that the unit ball of $D(\tilde{L}_p)$ is covered by the balls in $L^p(\R^d;\C^m)$ centered at $\overline g_j$ of radius $2^{\frac{1}{p}}\varepsilon$. As $\varepsilon >0$ was arbitrary, it follows that the unit ball of $D(\tilde{L}_p)$ is totally bounded
in $L^p(\R^d;\C^m)$. Therefore $D(\tilde{L}_p)$ is compactly embedded in $L^p(\R^d;\C^m)$ and hence $\tilde{L}_p$ has compact resolvent.

The spectral mapping theorem for the resolvent (cf. \cite[Theorem VI.1.13]{Nag}) and spectral properties of compact operators show that the spectrum $\sigma(\tilde{L}_p)$ of $\tilde{L}_p$ consists only of eigenvalues. Finally the $p$-independence of $\sigma(\tilde{L}_p)$ follows from \cite[Corollary 1.6.2]{Davies}, since the resolvent operators $(\lambda -\tilde{L}_p)^{-1}$ are consistent (see Proposition \ref{prop. positivity+consistence}.(1)) and compact.
\end{proof}
%\begin{coro}
%Let $1<p<\infty$ and assume that hypotheses of Proposition \ref{Prop Compactness} hold. Then, $\tilde{L}_p$ has a discrete spectrum which consists only of eigenvalues. Moreover, the spectrum is independent of $p\in(1,\infty)$.
%\end{coro}
%\begin{proof}
%The fact that the spectrum is discret and consists of eigenvalues is a consequence of the compactness of the resolvent.\textcolor{red}{For the $p$-independence one can just refer to the Davies book, like in the paper with Luca and Markus, since the resolvent is compact and consistent, but i choosed to use the proof -slightly modified- you have suggested in the 1st paper}\\
%Let us prove the $p$-independence. Let $1<p<q<\infty$, $\lambda\in \sigma(\tilde{L}_p)$ and $f\in D(\tilde{L}_p)$ such that $\tilde{L}_p f=\lambda f$. Hence, $\tilde{S}_p(t)f=e^{\lambda t}f$. So, by ultracontractivity, see Proposition \ref{Prop Lp-Lq estimate}, $f=e^{-\lambda t}\tilde{S}_p(t)f\in L^q(\R^d,\R^m)$ and by consistency $\tilde{S}_q(t)f=\tilde{S}_p(t)f=e^{\lambda t}f$. Follows easily that $f\in D(\tilde{L}_q)$ and $\tilde{L}_q f=\lambda f$. Then $\sigma(\tilde{L}_p)\subseteq\sigma(\tilde{L}_q)$. For the other inclusion, one has $\sigma(\tilde{L}_q)=\sigma(\tilde{L}_{q}^*)$. Arguing similarely by consistency and ultracontractivity of $\{\tilde{S}_p^*(t)\}$ and taking into account that $q'<p'$ one gets $\sigma(\tilde{L}_q^*)\subseteq\sigma(\tilde{L}_p^*)$ which ends the proof.
%\end{proof}
\subsection{Analyticity}
As it have been proved in \cite[Example 3.5]{KLMR}, Hypotheses \ref{Hyp. of KLMR} do not lead, in general, to the generation of an analytic semigroup. The following result yields a sufficient condition.
\begin{prop}
Let $1<p<\infty$. Assume that Hypotheses \ref{Hyp. of KLMR} and \eqref{grad v< v} are satisfied and there exists $M>0$ such that
\begin{equation}\label{cond. sectoriality}
\Re\langle -\tilde{V}(x)\xi,\xi \rangle\ge M\,|\Im\langle \tilde{V}(x)\xi,\xi \rangle|
\end{equation}
for all $x\in\R^d$ and $\xi\in\C^m$. Then, $\tilde{L}_p$ generates an analytic semigroup on $L^p(\R^d,\C^m)$.
\end{prop}
 \begin{proof}
 Let $u\in D(\tilde{L}_p)\subseteq W^{2,p}(\R^d,\C^m)$. According to \cite[Theorem 3.9]{Ouhabaz}, one has
 \[\Re\langle-D_p u,|u|^{p-2}u\rangle_{p,p'}\ge c_p|\Im\langle-D_p u,|u|^{p-2}u\rangle_{p,p'}|,\]
 where $D_p u=div(Q\nabla u)$ and $c_p=\frac{2\sqrt{p-1}}{|p-2|}$ if $p\neq 2$ and any positive constant if $p=2$. Then,
 \begin{align*}
 \Re\langle-\tilde{L}_p u,|u|^{p-2}u\rangle_{p,p'}&=\Re\langle-D_p u,|u|^{p-2}u\rangle_{p,p'}+Re\langle-\tilde{V}_p u,|u|^{p-2}u\rangle_{p,p'}\\
 &\ge c_p|\Im\langle-D_p u,|u|^{p-2}u\rangle_{p,p'}|+\int_{\R^d}\Re\langle-\tilde{V}(x)u(x),u(x)\rangle|u(x)|^{p-2}dx\\
 &\ge c_p|\Im\langle-D_p u,|u|^{p-2}u\rangle_{p,p'}|+M\int_{\R^d}|\Im\langle-\tilde{V}(x)u(x),u(x)\rangle||u(x)|^{p-2}dx\\
 & \ge M_p |\Im\langle (D_p+V_p) u,|u|^{p-2}u\rangle_{p,p'}|=M_p |\Im\langle \tilde{L}_p u,|u|^{p-2}u\rangle_{p,p'}|,
 \end{align*}
 where $M_p=\inf(c_p,M)$. This implies that $-\tilde{L}_p$ is sectorial of angle less than $\pi/2$ and then $\tilde{L}_p$ generates an analytic semigroup on $L^p(\R^d,\C^m)$.
 \end{proof}
 \begin{examps}
 The condition \eqref{cond. sectoriality} is satisfied for symmetric potential matrices but never for antisymmetric ones. Moreover, it has been proved in \cite[Example 4.5]{KLMR} that the semigroup generated by $L_p$ with the antisymmetric potential $V(x)=\begin{pmatrix}
   0 & -x\\
   x & 0
   \end{pmatrix}$ is not analytic.  However, we recover analyticity when perturbing $V$ by $(1+|x|^r)I_m$, for some $r\ge 1$.
  %Here we give an example where $\tilde{V}$ is not symmetric, however \eqref{cond. sectoriality} is verified.
   Indeed, consider $\tilde{V}:\R\to\R^{2\times 2}$ given by
 \[\tilde{V}(x)=\begin{pmatrix}
  -(1+|x|^r) & -x\\
 x & -(1+|x|^r)
 \end{pmatrix}=\begin{pmatrix}
 0 & -x\\
 x & 0
 \end{pmatrix}-(1+|x|^r)I_2,\]
 where $r\ge 1$. Let us show that $\tilde{V}$ verify \eqref{cond. sectoriality}. For $\xi=\begin{pmatrix}
  \xi_1\\
  \xi_2
  \end{pmatrix}\in \C^2$ one has
  \begin{align*}
  \langle\tilde{V}(x)\xi,\xi\rangle=-(1+|x|^r)(\xi_1^2+\xi_2^2)+x(\xi_1\bar{\xi}_2-\bar{\xi}_1\xi_2).
  \end{align*}
  Then, $$\Re\langle\tilde{V}(x)\xi,\xi\rangle=-(1+|x|^r)(\xi_1^2+\xi_2^2)$$ and $$\Im\langle\tilde{V}(x)\xi,\xi\rangle=x(\xi_1\bar{\xi}_2-\bar{\xi}_1\xi_2).$$
  Moreover, one has
  \[\left|\Im\langle\tilde{V}(x)\xi,\xi\rangle\right|\le 2|x||\xi_1\xi_2|\le (1+|x|^r)(\xi_1^2+\xi_2^2)=\Re\langle-\tilde{V}(x)\xi,\xi\rangle.\]
Hence \eqref{cond. sectoriality} holds for $\tilde{V}$.
 \end{examps}
\section{Kernel estimates}

We use the same notation and assume the same hypotheses of Section 4. We start by giving a generation result on $L^{1}(\R^d,\C^m)$ for a suitable realisation of $\tilde{\mathcal{L}}$ and an $L^p-L^q$-estimate for $\{\tilde{S}(t)\}_{t\geq 0}$. For the proof of the following two proposition one can see \cite[Theorem 3.7]{KLMR} and \cite[Section 4.2]{KLMR}.
  \begin{prop}\label{L_1 result}
  The restriction of $\{\tilde{S}_2(t)\}_{t\geq 0}$ to $L^{2}(\R^d,\C^m)\cap L^{1}(\R^d,\C^m)$ can be extended to a contractive $C_0$-semigroup $\{\tilde{S}_1(t)\}_{t\geq 0}$ on $L^{1}(\R^d,\C^m)$. Moreover,  $\{\tilde{S}_1(t)\}_{t\geq 0}$ is consistent with $\{\tilde{S}_p(t)\}_{t\geq 0}$, $1<p<\infty$, and the generator $\tilde{L}_1$ of $\{\tilde{S}_1(t)\}_{t\geq 0}$ coincides with all $\tilde{L}_p$'s, $1<p<\infty$, on $C_c^\infty(\R^d,\C^m)$.
  \end{prop}
  As a consequence of the consistency of the semigroups $\{\tilde{S}_p(t)\}_{t\geq 0}$, $1<p<\infty$, we drop the index $p$ and use $\{\tilde{S}(t)\}_{t\geq 0}$ to indicate our semigroup on $L^p(\R^d,\C^m)$.
   \begin{prop}\label{Prop Lp-Lq estimate}
    Let $1\leq p<q\leq\infty$. Then, for all $t>0$, $\tilde{S}(t)$ maps $L^{p}(\R^d,\C^m)$ into $L^{q}(\R^d,\C^m)$ and there exists a positive constant $\tilde{M}$ such that
   \begin{equation}
   \|\tilde{S}(t)f\|_{L^q(\R^d,\C^m)}\leq \tilde{M}t^{-\frac{d}{2}(\frac{1}{p}-\frac{1}{q})}\|f\|_{L^p(\R^d,\C^m)},\qquad t>0, \; f\in L^p(\R^d,\C^m).
   \end{equation}
   In particular, $\{\tilde{S}(t)\}_{t\geq 0}$ is ultracontractive, i.e.
   \begin{equation}\label{ultracontractivity ineq.}
   \|\tilde{S}(t)f\|_{\infty}\leq \tilde{M} t^{-\frac{d}{2}}\|f\|_{L^1(\R^d,\C^m)}
   \end{equation}
   for every $t>0$ and $f\in L^1(\R^d,\C^m)$.
   \end{prop}
\subsection{Gaussian estimates}
The immediate consequence of $(\ref{ultracontractivity ineq.})$ is that $\{\tilde{S}(t)\}_{t\geq 0}$ is given by a matrix kernel.
\begin{coro}\label{exist of kernel}
For all $t>0$ there exists $\tilde{K}(t,\cdot ,\cdot)=(\tilde{k}_{ij}(t,x,y))_{1\leq i,j\leq m}\in L^\infty(\R^d\times\R^d,\R^{m\times m})$ such that, for all $1<p<\infty$,
\begin{equation}\label{kernel representation}
\tilde{S}(t)f(x)=\int_{\R^d}\tilde{K}(t,x,y) f(y)dy,\quad f\in L^p(\R^d,\C^m).
\end{equation}
 Moreover,
\begin{equation}\label{boundedness of kernel}
 |\tilde{k}_{ij}(t,x,y)|\leq \tilde{M}t^{-\frac{d}{2}},\qquad t>0,\, x,y\in\R^d,\,1\leq i,j\leq m ,
\end{equation}
and, for any $t>0$, $\tilde{S}(t)$ is positive if, and only if, $\tilde{k}_{ij}(t,x,y)\geq 0$ for almost every $x,y\in\R^d$ and all $t>0$.
\end{coro}
\begin{proof}
 The existence of the kernel $\tilde{K}$ and \eqref{boundedness of kernel} are consequences of \eqref{ultracontractivity ineq.}. \\
 On the other hand, it is obvious that the positivity of all entries of the kernel matrix $\tilde{K}$ is a sufficient condition for positivity. Conversely, let $t>0$, $i,j\in\{1,\dots,m\}$ and $B$ any bounded measurable set of $\R^d$. Then
\begin{align*}
\tilde{S}(t)\geq 0 &\Longrightarrow \langle \tilde{S}(t)(\chi_B e_i),e_j\rangle \geq 0\\
&\Longrightarrow \int_B \tilde{k}_{ij}(t,x,y)dy \geq 0.
\end{align*}
As $B$ is arbitrary chosen, one gets $\tilde{k}_{ij}(t,x,y)\ge 0$ for a.e. $x,y\in\R^d$ and all $t>0$.
\end{proof}
We give now a Gaussian upper bound estimate for $\{\tilde{S}(t)\}_{t\geq 0}$. For the proof, we follow the strategy of \cite[Chapter 6]{Ouhabaz}
\begin{thm}\label{Thm Gauss ker est}
Assume that Hypotheses \ref{Hyp. of KLMR} and $0\le v\in W^{1,\infty}_{loc}(\R^d)$ satisfying \eqref{grad v< v}. Then there exist positive constants $C_1$ and $C_2$ such that
\begin{equation}\label{Gauss. ker. estimate}
|\tilde{k}_{ij}(t,x,y)|\leq C_1 t^{-\frac{d}{2}}\exp\{-C_2\dfrac{|x-y|^2}{4t}\}
\end{equation}
 for all $i,j\in\{1,\dots,m\},\,t>0$ and $ x,y\in\R^d$.
\end{thm}
\begin{proof}
 Let $\lambda\in\R$ and $\varphi\in E:=\{\psi\in C_b^\infty(\R^d) : \|\nabla\psi\|_\infty\leq 1\}$. Define the semigroup $\{\tilde{S}_{\lambda,\varphi}(t)\}_{t\geq 0}$ by
\[  \tilde{S}_{\lambda,\varphi}(t)f=e^{-\lambda\varphi}\tilde{S}(t)(e^{\lambda\varphi}f)\]
 for all $f\in L^p(\R^d,\C^m)$ with $1\leq p<\infty$. One has
\begin{equation}\label{kernel of twisted semgpe}
\tilde{S}_{\lambda,\varphi}(t)f=\int_{\R^d}e^{-\lambda(\varphi(x)-\varphi(y))}\tilde{K}(t,x,y) f(y)dy.
\end{equation}
Denote by $\tilde{L}_{\lambda,\varphi}$ the generator of $\{\tilde{S}_{\lambda,\varphi}(t)\}_{t\geq 0}$ on $L^2(\R^d,\C^m)$ and let $f\in C_c^\infty(\R^d,\R^m)$. A straightforward calculation yields
\[ \tilde{L}_{\lambda,\varphi} f=div(Q\nabla f)+2\lambda\langle Q\nabla\varphi,\nabla f\rangle+(\tilde{V}+\lambda div(Q\nabla\varphi)+\lambda^2 |\nabla\varphi|_Q^2)f. \]
Moreover,
\[ \langle-\tilde{L}_{\lambda,\varphi}f,f\rangle=-\langle(\tilde{L}_{\lambda,\varphi}-\tilde{V})f,f\rangle-\langle \tilde{V} f,f\rangle\geq -\langle(\tilde{L}_{\lambda,\varphi}-\tilde{V})f,f\rangle. \]
Integrating by parts, one obtains
\begin{align*}
-\langle(\tilde{L}_{\lambda,\varphi}-\tilde{V})f,f\rangle&= \sum_{i=1}^m\int_{\R^d}\langle Q(x)\nabla f_i(x),\nabla f_i(x)\rangle dx-2\lambda\sum_{i=1}^{m}\int_{\R^d}\langle Q(x)\nabla \varphi(x),\nabla f_i(x)\rangle f_i(x) dx\\
&\quad -\int_{\R^d}\{\lambda div(Q\nabla\varphi)(x)+\lambda^2|\nabla\varphi(x)|_{Q(x)}^2 \} |f(x)|^2 dx\\
&= \sum_{i=1}^m\int_{\R^d}\langle Q(x)\nabla f_i(x),\nabla f_i(x)\rangle dx-\lambda^2\int_{\R^d}\langle Q(x)\nabla \varphi(x),\nabla \varphi(x)\rangle |f(x)|^2 dx\\
&\geq \eta_1 \|\nabla f\|_2^2 -\eta_2 \lambda^2 \|f\|_2^2.
\end{align*}
If we set $\omega=\eta_2 \lambda^2$, then
\[ -\langle(\tilde{L}_{\lambda,\varphi}-\omega)f,f\rangle\geq \eta_1 \|\nabla f\|_2^2. \]
Consider now the function $\gamma(t)=\|e^{-\omega t}\tilde{S}_{\lambda,\varphi}(t)f\|_2^{-\frac{4}{d}}$ for all $t\geq 0$. So, one has
\begin{align*}
\gamma'(t)&= \frac{d}{dt}(\|e^{-\omega t}\tilde{S}_{\lambda,\varphi}(t)f\|_2^2)^{-\frac{2}{d}}\\
&=-\frac{4}{d}\|e^{-\omega t}\tilde{S}_{\lambda,\varphi}(t)f\|_2^{-\frac{4}{d}-2}\langle(\tilde{L}_{\lambda,\varphi}-\omega)e^{-\omega t}\tilde{S}_{\lambda,\varphi}(t)f,e^{-\omega t}\tilde{S}_{\lambda,\varphi}(t)f\rangle\\
&\geq \frac{4\eta_1}{d}\|e^{-\omega t}\tilde{S}_{\lambda,\varphi}(t)f\|_2^{-\frac{4}{d}-2}\|\nabla(e^{-\omega t}\tilde{S}_{\lambda,\varphi}(t)f)\|_2^2.
\end{align*}
By Nash's inequality, cf. \cite[Theorem 2.4.6]{Davies},
%\[ \|g\|_2^{2+\frac{4}{N}}\leq C\|\nabla g\|^2 \|g\|_1^{\frac{4}{N}},\qquad g\in W^{1,2}(\R^N,\R^m)\cap L^{1}(\R^N,\R^m), \]
one obtains
\[ \gamma'(t)\geq \frac{4\eta_1}{d\,C}\|e^{-\omega t}\tilde{S}_{\lambda,\varphi}(t)f\|_1^{-\frac{4}{d}}. \]
for some $C>0$. Since $A_{\lambda,\varphi}:=\tilde{L}_{\lambda,\varphi}-\omega-\tilde{V}$ is an elliptic operator with bounded coefficients and its associated form in $L^2(\R^d,\C^m)$ is accretive then $A_{\lambda,\varphi}$ generates a contractive semigroup in $L^1(\R^d,\C^m)$. Applying the Trotter-Kato product formula one obtains
$\|e^{-\omega t}\tilde{S}_{\lambda,\varphi}(t)f\|_1\leq \|f\|_1$. So, it follows that
\[\gamma(t)\geq \int_{0}^{t}\gamma'(s)ds\geq \frac{2\eta_1}{d\,C} t\|f\|_1^{-\frac{4}{d}}. \]
Hence,
 \begin{equation}\label{L^1-L^2 estimate}
 \|\tilde{S}_{\lambda,\varphi}(t)f\|_2\leq \left(\frac{4\eta_1}{d\,C}\right)^{-\frac{d}{4}}e^{\omega t} t^{-\frac{d}{4}}\|f\|_1.
 \end{equation}
 Since $V^*$ verifies the same hypotheses as $V$ (taking in consideration Remark \ref{adjoint-lp}), one obtains, by similar arguments as above,
 \[ \|\tilde{S}^*_{\lambda,\varphi}(t)f\|_2\leq \left(\frac{4\eta_1}{d\, C}\right)^{-\frac{d}{4}}e^{\omega t} t^{-\frac{d}{4}}\|f\|_1. \]
 Now, let $g\in C_c^\infty(\R^d,\C^m)$. One has
 \begin{align*}
 \left|\int_{\R^d}\langle\tilde{S}_{\lambda,\varphi}(t)f(x),g(x)\rangle dx\right| &=\left|\langle \tilde{S}_{\lambda,\varphi}(t)f,g\rangle_{L^2}\right| \\
 &=\left|\langle f,\tilde{S}^*_{\lambda,\varphi}(t)g\rangle_{L^2}\right| \\
 &\le\|\tilde{S}^*_{\lambda,\varphi}(t)g\|_2 \|f\|_2 \\
 &\le \left(\frac{4\eta_1}{d\, C}\right)^{-\frac{d}{4}}e^{\omega t} t^{-\frac{d}{4}}\|f\|_2\|g\|_1.
 \end{align*}
 Therefore,
 %$g\mapsto \int_{\R^d}\langle\tilde{S}_{\lambda,\varphi}(t)f(x),g(x)\rangle dx$ can be extended to a bounded form on $L^1(\R^d,\R^m)$ which has $L^\infty(\R^d,\R^m)$ as a dual space. Then
 $\tilde{S}_{\lambda,\varphi}(t)f\in L^\infty(\R^d,\C^m)$ and
 \begin{equation}\label{L2-L8 estimate}
 \|\tilde{S}_{\lambda,\varphi}(t)f\|_\infty\leq \left(\frac{4\eta_1}{d\, C}\right)^{-\frac{d}{4}}e^{\omega t} t^{-\frac{d}{4}}\|f\|_2.
 \end{equation}
 Combining $(\ref{L^1-L^2 estimate})$ and $(\ref{L2-L8 estimate})$ one obtains $\tilde{S}_{\lambda,\varphi}(t)f\in L^\infty(\R^d,\C^m)$ for every $f\in L^1(\R^d,\C^m)$ and
 \begin{align*}
 \|\tilde{S}_{\lambda,\varphi}(t)f\|_\infty =\|\tilde{S}_{\lambda,\varphi}(t/2)\tilde{S}_{\lambda,\varphi}(t/2)f\|_\infty&\le
 \left(\frac{4\eta_1}{d\, C}\right)^{-\frac{d}{4}}e^{\omega t/2} (t/2)^{-\frac{d}{4}}\|S_{\lambda,\varphi}(t/2)f\|_2\\
 &\le C_1 e^{\omega t} t^{-\frac{d}{2}}\|f\|_1,
 \end{align*}
 with $C_1=2^{d/2}\left(\frac{4\eta_1}{d\, C}\right)^{-\frac{d}{2}} $.\\
 Arguing similarly as in Corollary $\ref{exist of kernel}$ and taking into account \eqref{kernel of twisted semgpe} one gets
 \[ |\tilde{k}_{ij}(t,x,y)|\leq C_1 t^{-\frac{d}{2}}\exp\{\eta_2\lambda^2 t+\lambda(\varphi(x)-\varphi(y))\},\quad t>0,\,x,y\in \R^d. \]
 Thanks to the arbitrariness of $\lambda$ one can choose  $\lambda=\dfrac{\varphi(y)-\varphi(x)}{2\eta_2 t}$ and obtains
 \[ |\tilde{k}_{ij}(t,x,y)|\leq C_1 t^{-\frac{d}{2}}\exp\{-\dfrac{|\varphi(x)-\varphi(y)|^2}{4\eta_2 t}\}. \]
 Note that the distance $\delta$ on $\R^d$ defined by
 \[ \delta(x,y)=sup\{\psi(x)-\psi(y) : \psi\in E\},\qquad x,y\in\R^d, \]
  is equivalent to the euclidian distance in $\R^d$. Therefore, there exists $C_2>0$ such that
 \[ |\tilde{k}_{ij}(t,x,y)|\leq C_1 t^{-\frac{d}{2}}\exp\{-C_2\dfrac{|x-y|^2}{4t}\}. \]
 \end{proof}
 The following example shows how scalar Schr\"odinger operators can be seen as Schr\"odinger systems with real matrix potentials.
 \begin{examp}\label{exa-complexe}
 Let us consider the matrix potential
 $$\tilde{V}(x):=\begin{pmatrix}
-v(x) & -w(x)\\
w(x) & -v(x)
\end{pmatrix}=w(x)\begin{pmatrix}
0 & -1\\
1 & 0
\end{pmatrix}-v(x)\begin{pmatrix}
1 & 0\\
0 & 1
\end{pmatrix},$$
where $w(x)=1+|x|^r$ and $v(x)=1+|x|^\alpha,\,x\in \R^d$, with $r\in [1,2)$ and $\alpha \ge 1$. Taking into account Remark \ref{V-dissip} we deduce,
by Corollary \ref{coro gen. of s.g.}, that the operator
$$\tilde{L}_p=\begin{pmatrix}
\Delta & 0\\
0 & \Delta
\end{pmatrix}+\tilde{V}\quad \hbox{\ with domain }W^{2,p}(\R^d,\C^2)\cap D(|x|^r)\cap D(|x|^\alpha)$$
generates a $C_0$-semigroup on $L^p(\R^d,\C^2)$. Applying Corollary \ref{exist of kernel} and Theorem \ref{Thm Gauss ker est} we know that this semigroup is a given by a kernel satisfying Gaussian estimates.\\
Now, we diagonalize the matrix $\begin{pmatrix}
0 & -1\\
1 & 0
\end{pmatrix}$ and so we obtain that $\tilde{L}_p$ is similar to the operator
$$P^{-1}\tilde{L}_pP=\begin{pmatrix}
\Delta & 0\\
0 & \Delta
\end{pmatrix}+\begin{pmatrix}
i(1+|x|^r)-(1+|x|^\alpha) & 0\\
0 & -i(1+|x|^r)-(1+|x|^\alpha)
\end{pmatrix},$$
where $P=\begin{pmatrix}
1 & 1\\
-i & i
\end{pmatrix}$.
Thus the following Schr\"odinger operators with complex potentials $\Delta \pm i(1+|x|^r)-(1+|x|^\alpha)$ with domain $W^{2,p}(\R^d)\cap D(|x|^\alpha+i|x|^r)$ generates $C_0$-semigroups on $L^p(\R^d)$ and satisfy Gaussian estimates.
\end{examp}
 \subsection{Further kernel estimates}
In this subsection we compare the off-diagonal entries $\tilde{k}_{ij},\,i\neq j$, with $\tilde{k}_{ii}$ for all $i,j\in \{1,\ldots m\}$, and deduce more precise kernel estimates with respect to space variables in the symmetric case.

Here, in addition to Hypotheses \ref{Hyp. of KLMR}, we assume the following
\begin{hyp}\label{Hyp. of OURH}
%\begin{itemize}
$V$ is symmetric and $v(x)=1+|x|^\alpha,\,x\in \R^d$, with $\alpha \ge 1$.
%\item[(ii)] $q_{ij}\in W^{1,\infty}_{loc}(\R^d)$ and $\partial_jq_{ij}=o(|x|^{\frac{\alpha}{2}})$.
%\end{itemize}
\end{hyp}
Let us start with the following comparison result.
\begin{lm}
The kernel $\tilde{K}$ of $\{\tilde{S}(t)\}_{t\ge 0}$ satisfies
\begin{equation}\label{dia-nondia}
|\tilde{k}_{ij}(t,x,y)+\tilde{k}_{ij}(t,y,x)|\le 2\left(\tilde{k}_{ii}(t,x,y)\tilde{k}_{jj}(t,x,y)\right)^{\frac{1}{2}}
\end{equation}
for all $i,j\in \{1,\ldots ,m\},\,t>0$ and a.e. $x,y\in \R^d$.
\end{lm}
\begin{proof}
Similar to the proof of the positivity of the kernels $\tilde{k}_{ij}$, one can deduce, using the symmetry of $\{\tilde{S}(t)\}_{t\ge 0}$, that
$\tilde{k}_{ij}(t,x,y)=\tilde{k}_{ji}(t,y,x)$ for all $i,j\in \{1,\ldots ,m\}$ and
$$\langle \tilde{K}(t,x,y)\xi,\xi \rangle \ge 0$$
for all $t>0,\,\xi \in \R^m$ and a.e. $x,y\in \R^d$. On the other hand, it can be seen that every positive matrix $M:=(m_{ij})$ in $\R^m$ satisfies
$$|m_{ij}+m_{ji}|\le 2\sqrt{m_{ii}m_{jj}},\quad \forall i,j\in \{1,\ldots ,m\}.$$
Thus, \eqref{dia-nondia} follows.
\end{proof}
Let us denote by $k_v$ the heat kernel of the semigroup generated by the scalar Schr\"odinger operator $L_v:=div(Q\nabla \cdot)-v$.\\
We can now deduce upper bounds for the kernels $\tilde{k}_{ij}$ from the ones of $k_v$.
\begin{thm}
Assume Hypotheses \ref{Hyp. of KLMR} and \ref{Hyp. of OURH} with $\alpha >2$. Then for $\theta >0$ such that $\theta \eta_2<1$ we have
 %For all $i,j\in\{1,\dots,m\}$, one has
\begin{equation}\label{Ouha-Rha ker-est}
|\tilde{k}_{ij}(t,x,y)+\tilde{k}_{ij}(t,y,x)|\le Ce^{-\mu_0 t}e^{ct^{-b}}\frac{1}{(|x||y|)^\beta}e^{-\frac{\sqrt{\theta}}{\gamma}|x|^\gamma}e^{-\frac{\sqrt{\theta}}{\gamma}|y|^\gamma}
\end{equation}
for $|x|,\,|y|\ge 1,\,t>0$ and $i,j\in\{1,\dots,m\}$. Here $c,\,C>0,\,b>\frac{\alpha+2}{\alpha-2},\,\beta=\frac{\alpha}{4}+\frac{d-1}{2},\,\gamma=1+\frac{\alpha}{2}$ and $\mu_0$ is the first eigenvalue of $L_v$. %$$\phi(R)=\dfrac{\exp\left(-\frac{2\sqrt{\theta}}{2+\alpha}R^{1+\frac{\alpha}{2}}\right)}{R^{\frac{\alpha}{4}+\frac{d-1}{2}}},$$ for every $R\ge 0$. Here  $b>\frac{\alpha+2}{\alpha-2}$ and $\theta>0$ is properly choosen.
\end{thm}
\begin{proof}
One has to note first that $\tilde{k}_{ii}$ is the heat kernel of the semigroup generated by the scalar operator $\tilde{L}_v:= div(Q\nabla \cdot)-(v-v_{ii})$ and, by \eqref{diss of V} , the potential $v-v_{ii}\ge v$. So, the assertion follows from \eqref{dia-nondia} and \cite[Theorem 2.7]{Ouha-Rha}.
%Combining \eqref{kernel < scalar kernel} and \eqref{x alpha < V < x alpha} one obtains $\tilde{k}_{ij}(t,x,y)\le e^{d_1 t}k_{c_1,\alpha}(t,x,y)$, where $k_{c_1,\alpha}$ is the kernel associated to the scalar Schr\"odinger with potential $c_1|x|^\alpha$. Here $c_1>0$ and $d_1>0$. Thus \cite[Theorem 2.7]{Ouha-Rha} yields \eqref{Ouha-Rha ker-est}. The constant $\theta$ is such that $\theta\eta_2<1$, where $\eta_2$ is the constant appearing in \eqref{ell}.
\end{proof}

A lower kernel estimate holds for entries of the matrix kernel $\tilde{K}(\cdot,\cdot,\cdot)$ in the case where the semigroup $\{\tilde{S}(t)\}_{t\ge0}$ is positive.
\begin{prop}\label{lower-dom}
  Assume Hypotheses \ref{Hyp. of KLMR} and \ref{Hyp. of OURH}. If $0\le-v_{ii}\le v$ and $v_{hl}\ge 0$, for all $h\neq l\in\{1,\dots,m\}$, then
 \begin{equation}\label{eq. S(t)> S_2v(t)}
 0\le S_{2v}(t)f\le \langle\tilde{S}(t)(fe_i),e_j\rangle
 \end{equation}
 for every $t>0$, $0\le f\in C_c^\infty(\R^d)$ and $i,j\in\{1,\dots,m\}$. Here $\{S_{2v}(t)\}_{t\ge0}$ is the $C_0$-semigroup generated by the scalar Schr\"odinger operator $L_{2v}:=div(Q\nabla\cdot)-2v$ in $L^2(\R^d)$. In particular,
 \begin{equation}\label{eq. k_ij > k_2v}
  0\le k_{2v}(t,x,y)\le  \tilde{k}_{ij}(t,x,y)
  \end{equation}
  for every $i,j\in\{1,\dots,m\}$, $t>0$ and $x,y\in\R^d$.
 \end{prop}
\begin{proof}
Fix $0\le f\in C_c^\infty(\R^d)$, $i,j\in\{1,\dots,m\}$ and $t>0$. Consider $\varphi(s,\cdot)=S_{2v}(s)f$ and $\psi(s,\cdot)=\langle\tilde{S}(t-s)(\varphi(s,\cdot)e_i),e_j\rangle$ for $s\in[0,t]$. Since $f\in C_c^\infty(\R^d)\subset D(L_{2v})$, it follows that $\varphi(s,\cdot)=S_{2v}(s)f\in D(L_{2v})=W^{2,2}(\R^d)\cap D(2v)$ for all $s\in[0,t]$. Then, $\varphi(s,\cdot)e_i\in W^{2,2}(\R^d,\R^m)$. Let us now show that $\varphi(s,\cdot)e_i\in D(\tilde{V})$. This follows from the fact that $v_{ki}\le\sqrt{-v_{kk}}\sqrt{-v_{ii}}\le v,\,i\neq k$, which implies that $v_{ki}\varphi(s,\cdot)\in L^2(\R^d)$, for every $k\in\{1,\dots,m\}$. Hence $\varphi(s,\cdot)e_i\in D(\tilde{L})=W^{2,2}(\R^d,\R^m)\cap D(\tilde{V})$. %Therefore, $s\mapsto \psi(s,\cdot)$ is continuous in $[0,t]$ and, since the semigroup $\{\tilde{S}(.)\}$ is analytic, then $s\mapsto \psi(s,\cdot)$ is differentiable in $[0,t)$. Moreover, and one has
Differentiating the function $\psi$ with respect to $s\in [0,t]$ we have
 \begin{align*}
 \psi'(s,\cdot)&=\langle -\tilde{S}(t-s)\tilde{L}(\varphi(s,\cdot)e_i)+\tilde{S}(t-s)(\varphi'(s,\cdot)e_i),e_j\rangle\\
 &= \langle\tilde{S}(t-s)\left(\varphi'(s,\cdot)e_i-\tilde{L}(\varphi(s,\cdot) e_i)\right),e_j\rangle.
 \end{align*}
 On the other hand,
 \begin{align*}
 \varphi'(s,\cdot)e_i-\tilde{L}(\varphi(s,\cdot) e_i) &= \left(div(Q\nabla\varphi(s,\cdot))-2v\varphi(s,\cdot)\right)e_i\\ &-\left(div(Q\nabla\varphi(s,\cdot))e_i-v\varphi(s,\cdot)e_i+\varphi(s,\cdot)\sum_{l=1}^{m}v_{il}e_l\right)\\
 &=-(v+v_{ii})\varphi(s,\cdot) e_i-\varphi(s,\cdot)\sum_{l\neq k}v_{il}e_l\\
 &\le 0.
 \end{align*}
 %Now, \eqref{mu< -V < nu} implies that $\nu+v_{ii}\ge 0$ and by hypothesis $v_{il}\ge 0$ for $i\neq l$, thus $ \varphi'(s,\cdot)e_i-\tilde{L}(\varphi(s,\cdot) e_i)\le 0$.
 Since, by assumptions, the semigroup $\{\tilde{S}(t)\}_{t\ge0}$ is positive, it follows that $\psi '(s,\cdot)\le 0$ for all $s\in [0,t]$. Hence
 $\psi(t,\cdot)\le\psi(0,\cdot)$ and thus \eqref{eq. S(t)> S_2v(t)} follows\\
 %Note that $-v_{ii}\le v$ and $\varphi(s,\cdot)=S_{2v}(s)f\ge 0$ since the scalar semigroup $\{S_{2v}(t)\}_{t\ge0}$ is positive.
 %Since $v_{hl}\ge 0$, for $h\neq l$, then, according to Proposition \ref{prop. positivity+consistence}, $\tilde{S}(t-s)$ is positive. Therefore, $\psi'(s,\cdot)\le 0$, which yields $\psi(t,\cdot)\le\psi(0,\cdot)$ and thus \eqref{eq. S(t)> S_2v(t)} follows.
 Now, \eqref{eq. k_ij > k_2v} follows by taking into account that $\tilde{k}_{ij}(t,\cdot,\cdot)$ is the kernel associated to $\langle\tilde{S}(t)(\cdot e_i),e_j\rangle$, for every $t>0$.
 \end{proof}
 As a consequence we obtain by applying \cite[Theorem 6.3 and Theorem 3.2]{davies-simon84} the following lower estimates.
 \begin{coro}
 Assume the same assumptions as in Proposition \ref{lower-dom}. Moroever, assume that $\alpha>2$ and $Q=I_d$. Then we have
 $$\tilde{k}_{ij}(t,x,y)\ge c_t\frac{1}{(|x||y|)^\beta}e^{-\frac{|x|^\gamma}{\gamma}}e^{-\frac{|y|^\gamma}{\gamma}}$$
 for $|x|,\,|y|\ge 1,\,t>0$ and $i,j\in\{1,\dots,m\}$, where $c_t$ is a positive constant, $b>\frac{\alpha+2}{\alpha-2},\,\beta=\frac{\alpha}{4}+\frac{d-1}{2}$ and $\gamma=1+\frac{\alpha}{2}$.
 \end{coro}
 
\section{Asymptotic distribution of the eigenvalues of $-\tilde{L}$}
In this section we assume that Hypotheses \ref{Hyp. of KLMR} and \ref{Hyp. of OURH} are satisfied and consider matrix potentials $V$ with polynomial entries, or more precisely positive powers of $|x|$. On the other hand, the assumption \eqref{Cond. on V} applied for potentials of the form $|x|^\beta V_0$, with $V_0$ a suitable constant matrix, implies $\beta<2$. For this reason, we will assume that
\begin{equation}\label{eq-lei}
v_{ii}=o(|x|^\alpha), \hbox{\ as }|x|\to \infty ,\quad \forall i\in\{1,\dots,m\}.
\end{equation}
%%%%%%%%%%%%%%
%So, we have the following auxiliary result.
%\begin{lm}
%For every $\varepsilon>0$ there exists $C_\varepsilon>0$ such that
%\begin{equation}\label{x alpha < V < x alpha}
%((1-\varepsilon)|x|^\alpha-C_\varepsilon)|\xi|^2\le \langle -\tilde{V}(x)\xi,\xi\rangle \le ((1+\varepsilon)|x|^\alpha+C_\varepsilon)|\xi|^2,
%\end{equation}
%for all $x\in\R^d$ and $\xi\in\R^m$.
%\end{lm}
%\begin{proof}
%An equivalent way of the statement is  $(1-\varepsilon)|x|^\alpha-C_\varepsilon\le \lambda(x)\le (1+\varepsilon)|x|^\alpha+C_\varepsilon$, for every %eigenvalue $\lambda(x)$ of $-\tilde{V}(x)$. Fix, $x\in\R^d$ and $\lambda(x)\in\sigma(-\tilde{V}(x))$. Let $0\neq P(x)\in\R^m$ be such that %$-\tilde{V}(x)P(x)=\lambda(x)P(x)$. Let $|P_i(x)|=\max_{j}|P_j(x)|\neq 0$. Then,
%\begin{align*}
%\lambda(x)P_i(x)&= (-\tilde{V}(x)P(x))_i
%= (1+|x|^\alpha) P_i(x)-\sum_{j=1}^{m}v_{ij}(x)P_j(x),
%\end{align*}
%thus, \[ \lambda(x)=1+|x|^\alpha-\sum_{j=1}^{m}v_{ij}(x)\frac{P_j(x)}{P_i(x)}=|x|^\alpha+o(|x|^2) .\]
%Thus the claim follows.
%Hence, for every $\varepsilon>0$, there exists $C_\varepsilon>0$, such that
%\[ (1-\varepsilon)|x|^\alpha-C_\varepsilon\le \lambda(x)\le (1+\varepsilon)|x|^\alpha+C_\varepsilon.\]
%\end{proof}
%%%%%%%%%%%%%%%%%%%
Since $v(x)=1+|x|^\alpha$, it follows from Proposition \ref{Prop Compactness} that
$\sigma(-\tilde{L})=\{\lambda_n: n\in\N\}$ consists of eigenvalues only, and the set of corresponding eigenvectors $\{\Psi_n: n\in\N\}$ forms an orthonormal basis of $L^2(\R^d,\R^m)$. In the following proposition we compute the trace of $\{\tilde{S}(t)\}_{t\ge 0}$.
\begin{prop}
For all $i,j\in\{1,\dots,m\}$, one has
\begin{equation}\label{kernel=sum eigenvectors}
\tilde{k}_{ij}(t,x,y)=\sum_{n\in\N}e^{-\lambda_n t}\Psi_n^{(i)}(x)\Psi_n^{(j)}(y)
\end{equation}
for all $x,y\in\R^d$ and $t>0$. Here $\Psi_n^{(i)}(x)$ is the \emph{i}-th component of the vector $\Psi_n(x)$. In particular,
\begin{equation}\label{trace of semigroup}
\int_{\R^d}\sum_{i=1}^{m}\tilde{k}_{ii}(t,x,x)dx=\sum_{n\in\N}e^{-\lambda_n t},\qquad\forall t>0.
\end{equation}
\end{prop}
\begin{proof}
Let $f\in L^2(\R^d,\R^m)$. So, $f=\displaystyle\sum_{n\in\N}\langle f,\Psi_n\rangle_{L^2}\Psi_n$. Then,
\begin{align*}
\tilde{S}(t)f=\sum_{n\in\N}\langle f,\Psi_n\rangle_{L^2}\tilde{S}(t)\Psi_n=\sum_{n\in\N}\langle f,\Psi_n\rangle_{L^2} e^{-\lambda_n t}\Psi_n.
\end{align*}
for every $t>0$. Hence,
\begin{align*}
\langle\tilde{S}(t)f(x),e_i\rangle&= \sum_{n\in\N}e^{-\lambda_n t}\int_{\R^d}\sum_{j=1}^{m}f_j(y)\Psi_n^{(j)}(y) \Psi_n^{(i)}(x)\, dy
\end{align*}
for each $i\in\{1,\dots,m\}$. Therefore, for every $\varphi\in C_c^\infty(\R^d)$,
\begin{align*}
\int_{\R^d}\tilde{k}_{ij}(t,x,y)\varphi(y)dy&=\langle\tilde{S}(t)(\varphi e_j)(x),e_i\rangle\\
&=\int_{\R^d}\sum_{n\in\N}e^{-\lambda_n t}\Psi_n^{(j)}(y)\Psi_n^{(i)}(x)\varphi(y)\, dy,
\end{align*}
for all $t>0$, $x\in\R^d$ and $i,j\in\{1,\dots,m\}$. From which we deduce \eqref{kernel=sum eigenvectors}. Moreover,
\begin{align*}
\sum_{i=1}^{m}\int_{\R^d}\tilde{k}_{ii}(t,x,x)\,dx &=\sum_{i=1}^m\int_{\R^d}\sum_{n\in\N}e^{-\lambda_n t}\Psi_n^{(i)}(x)^2\, dx\\
&=\sum_{n\in\N}e^{-\lambda_n t}\int_{\R^d}|\Psi(x)|^2\, dx\\
&=\sum_{n\in\N}e^{-\lambda_n t}.
\end{align*}
\end{proof}
Let us now introduce the measure $\mu$ defined over $\R^+$ by $\mu(X)=|\{n: \lambda_n\in X\}|$. Define, for $\lambda>0$, $\mathcal{N}(\lambda)=\mu[0,\lambda]$ the number of $\lambda_n$ which are less or equal than $\lambda$. Let us denote by $\hat{\mu}$ the Laplace transform  of $\mu$, $\hat{\mu}(t):=\int_\R e^{-tx}d\mu(x)$, for all $t>0$. According to \eqref{trace of semigroup}, one has
\[\hat{\mu}(t)=\sum_{n\in\N}e^{-\lambda_n t}=\int_{\R^d}\sum_{i=1}^{m}\tilde{k}_{ii}(t,x,x)dx .\]
We are looking for the asymptotic behaviour of $\mathcal{N}(\lambda)$ when $\lambda\to\infty$. This is related to the behaviour near $0$ of $\hat{\mu}$ by the famous Tauberian theorem due to Karamata, cf. \cite[Theorem 10.3]{Simon}, \cite[Theorem 7.1]{Meta-Spina}. For the proof of the following theorem we use the same approach as in \cite[Section 4]{Meta-Spina}.
\begin{thm}
Assume that $Q=I_d$, \eqref{eq-lei}, Hypotheses \ref{Hyp. of KLMR} and \ref{Hyp. of OURH} are satisfied. Then,
\begin{equation}\label{asymtotic distribution of eigenvalues}
\lim_{\lambda\to\infty}\dfrac{\mathcal{N}(\lambda)}{\lambda^{d(\frac{1}{2}+\frac{1}{\alpha})}}=
\frac{1}{\alpha}\frac{m\,d\,\omega_d}{(4\pi)^\frac{d}{2}}\dfrac{\Gamma(d/\alpha)}{\Gamma(d(\frac{1}{2}+\frac{1}{\alpha})+1)},
\end{equation}
where $\omega_d$ denotes the volume of the unit sphere of $\R^d$.
\end{thm}
\begin{proof}
We recall that $\tilde{k}_{ii}$ is the heat kernel of the semigroup generated by the scalar operator $\tilde{L}_v=\Delta -(v-v_{ii})$. By \eqref{eq-lei} we know that $(v-v_{ii})(x)=|x|^\alpha +o(|x|^\alpha)$. So, the assertion follows from \cite[Proposition 4.4]{Meta-Spina}.
\end{proof}

\end{document}